\documentclass{article}
\usepackage[T1]{fontenc}
\usepackage{amssymb}
\usepackage{amsfonts}
\usepackage{amsmath,amsthm}
\usepackage{graphicx}
\usepackage{tikz}
\usepackage{flafter}

\newtheorem{theorem}{Theorem}

\newtheorem{corollary}{Corollary}

\newtheorem{definition}{Definition}

\newtheorem{lemma}{Lemma}
\newtheorem{notation}[theorem]{Notation}

\newtheorem{proposition}{Proposition}

\newcommand{\QED}{\hfill {\qedhere}}

\begin{document}

\title{The Title of a Standard LaTeX Article}
\author{A. U. Thor \\
%EndAName
The University of Stewart Island}

\begin{center}
\textbf{\large Nearly Spectral Spaces}\bigskip

Lorenzo Acosta G.\footnote{%
Mathematics Department, Universidad Nacional de Colombia, AK 30 45-03, Bogot%
\'{a}, Colombia. e-mail: lmacostag@unal.edu.co} and I. Marcela Rubio P.%
\footnote{%
Corresponding author. Mathematics Department, Universidad Nacional de
Colombia, AK 30 45-03, Bogot\'{a}, Colombia. e-mail: imrubiop@unal.edu.co}

\bigskip
\end{center}

\begin{quote}
\textbf{Abstract: }We study some natural generalizations of the spectral
spaces in the contexts of commutative rings and distributive lattices. We
obtain a topological characterization for the spectra of commutative (not
necessarily unitary) rings and we find spectral versions for the up-spectral
and down-spectral spaces. We show that the duality between distributive
lattices and Balbes-Dwinger spaces is the co-equivalence associated to a
pair of contravariant right adjoint functors between suitable categories.


\end{quote}

\section{Introduction}

A \textit{spectral space} is a topological space that is homeomorphic to the
prime spectrum of a commutative unitary ring. This type of spaces were
topologically characterized by Hochster \cite{Hochster} as the \textit{sober}%
, \textit{coherent} and \textit{compact} spaces. On the other hand, it is
known that a topological space is a spectral space if and only if it is
homeomorphic to the prime spectrum of a distributive bounded lattice \cite%
{Simmons}, \cite{Acosta}.

Therefore, this notion has two natural generalizations: the first in the
context of rings and the second in the context of lattices:

We say that:

(1) a topological space is \textit{almost-spectral} if it is homeomorphic to
the prime spectrum of a commutative (not necessarily unitary) ring,

(2) a topological space is a \textit{Balbes-Dwinger space} if it is
homeomorphic to the prime spectrum of a distributive (not necessarily
bounded) lattice.\footnote{%
In \cite{Balbes-D} this type of spaces are called \textit{Stone spaces}.
However, in several other references, for example \cite{Johnstone}, a Stone
space is a compact, Hausdorff and totally disconnected space.}

In Chapter VI of \cite{Balbes-D}, there is a topological characterization of
the Balbes-Dwinger spaces (called there Stone spaces). As far as we know, in
the literature, there is no topological characterization for the
almost-spectral spaces.

Furthermore, there exist generalizations on a topological point of view \cite%
{Echi 1}, \cite{Echi 2}:

(3) a topological space is called \textit{up-spectral} if it is sober and
coherent,

(4) a topological space is called \textit{down-spectral} if it is coherent,
compact and every proper irreducible closed set is the closure of a unique
point.

It is natural to ask if the notions in (3) and (4) have \textquotedblleft
spectral versions\textquotedblright , that is, if the corresponding spaces
are homeomorphic to prime spectra of some kind of rings or lattices.

In this paper we show that all these topological spaces are particular cases
of certain class of topological spaces (named here \textit{RA-spaces}) and
we give spectral versions for all of them. In addition, we give a
topological characterization of the almost-spectral spaces and a new,
simpler, topological characterization of the Balbes-Dwinger spaces.

Actually, we extend the co-equivalence (or duality) between the category of
distributive bounded lattices and the category of spectral spaces presented
in \cite{Balbes-D} to a pair of contravariant, adjoint functors between the
category of distributive lattices and the category of RA-spaces. By means of
this adjunction, all the mentioned types of topological spaces arise
naturally and the relationship between them becomes clear. In particular, we
can easily deduce the duality between up-spectral and down-spectral spaces
studied in \cite{Echi 2}.

\section{Preliminaries}

We recall some basic definitions and facts that will be useful in the next
sections.

\begin{notation}
If $g:X\rightarrow Y$ is a function, we denote $g^{\ast }$ the inverse image
function defined by 
\begin{equation*}
g^{\ast }:\mathcal{\wp }(Y)\rightarrow \mathcal{\wp }(X):B\mapsto g^{\ast
}\left( B\right) =\left\{ x\in X:g\left( x\right) \in B\right\} .
\end{equation*}
\end{notation}

\subsection{Lattice theory notions}

A \textit{lattice} is a non empty partially ordered set (or poset) such that
every pair of elements $a,b$ has least upper bound (or join) $a\vee b,$ and
greatest lower bound (or meet) $a\wedge b.$ The lattice is \textit{%
distributive} if $\vee $ is distributive with respect to $\wedge $
(equivalently $\wedge $ is distributive with respect to $\vee $). The
lattice is \textit{bounded} if it has least (or minimum) and greatest (or
maximum) elements, usually denoted $0$ and $1,$ respectively. An ideal of a
lattice is a non empty lower subset that is closed under finite (non empty)
joins. A proper ideal $I$ is prime if $a\wedge b\in I$ implies $a\in I$ or $%
b\in I.$

A map $\alpha :L\rightarrow M$ between lattices is a \textit{homomorphism}
if for each pair of elements $a,b\in L$, $\alpha (a\wedge b)=\alpha
(a)\wedge \alpha (b)$ and $\alpha (a\vee b)=\alpha (a)\vee \alpha (b).$ The
homomorphism $\alpha $ is \textit{proper} if the inverse image of any prime
ideal of $M$ is a prime ideal of $L.$

\noindent The \textit{prime spectrum of a lattice} $L$ is the set of its
prime ideals endowed with the \textit{Zariski (or hull-kernel) topology},
whose basic open sets are the sets 
\begin{equation*}
d(a)=\left\{ I:I\text{ is a prime ideal of }L\text{ and }a\notin I\right\} ,
\end{equation*}%
where $a\in L.$ We denote this space by $\mathfrak{spec}(L).$ Actually, $%
d:L\rightarrow \wp \left( \mathfrak{spec}(L)\right) $ is a homomorphism of
lattices such that $d\left( 0\right) =\emptyset ,$ when $L$ has minimum and $%
d\left( 1\right) =\mathfrak{spec}(L),$ when $L$ has maximum. This
homomorphism is injective if and only if the lattice $L$ is distributive. It
is known that for each $a\in L$, $d\left( a\right) $ is a compact subspace
of $\mathfrak{spec}(L)$.

\subsection{Ring theory notions}

Similarly, the \textit{prime spectrum of a commutative ring} $A$ is defined
as the set of its prime ideals endowed with the \textit{Zariski (or
hull-kernel) topology}, whose basic open sets are the sets 
\begin{equation*}
D(a)=\left\{ P:P\text{ is a prime ideal of }A\text{ and }a\notin P\right\} ,
\end{equation*}%
where $a\in A.$ In this case the closed sets are%
\begin{equation*}
V(I)=\{P:P\text{ is a prime ideal of }A\text{ and }P\supseteq I\},
\end{equation*}%
where $I$ is an ideal of $A.$ We denote this space by $Spec(A),$ as usual.
Notice that $D:A\rightarrow \wp \left( Spec(A)\right) $ is such that for
each $a,b\in A,$ $D(ab)=D(a)\cap D(b)$ and $D(a+b)\subseteq D(a)\cup D(b).$
It is also known that the basic open sets are compact. Therefore, the prime
spectrum of a commutative unitary ring is a compact topological space;
however, compactness of $Spec(A)$ is not equivalent to existence of identity
in $A$. The following theorem, taken from \cite{Acosta-Rubio}, is useful:

\begin{theorem}
\label{nilcompactacion} Let $S$ be a commutative ring.

(i) If $R$ is a commutative ring such that $S$ is an ideal of $R,$ then $%
Spec(S)$ is homeomorphic to the open subspace $V(S)^{c}$ of $Spec(R).$

(ii) There exists a commutative unitary ring $Q(S)$ such that $Spec(S)$ is
homeomorphic to an open-dense subspace of $Spec(Q(S)).$
\end{theorem}

Another known fact is that for each ideal $I$ of the ring $A$ the function 
\begin{equation*}
\theta :V(I)\rightarrow Spec\left( A/I\right) :P\mapsto P/I
\end{equation*}
is a homeomorphism \cite{Atiyah}.

\subsection{Topological notions}

A subset $F$ of a topological space is an \textit{irreducible closed set} if 
$F$ is a non-empty closed set such that for every pair of closed sets $G$
and $H,$ $F=G\cup H$ implies $F=G$ or $F=H.$ We say that $U$ is a \textit{%
prime open set} if its complement is an irreducible closed set.\textbf{\ }

A space is called \textit{sober} if every irreducible closed set is the
closure of a unique point.

A space is called \textit{coherent} if it has a basis of open-compact sets
that is closed under finite intersections.

For example, an infinite set $X$ endowed with the co-finite topology is
coherent, but it is not sober since $X$ is an irreducible closed set that is
not the closure of any point. Notice that, in this example, all \textit{%
proper} irreducible closed sets are, in fact, closures of points.

We give then the following definition:

\begin{definition}
A topological space is \textbf{almost-sober }if every proper irreducible
closed set is the closure of some point\footnote{%
This notion is not taken from the literature. The notions of \textit{%
semi-sober} and \textit{quasi-sober} are found for example in \cite{Echi 2}\
and \cite{Echi 3}\ respectively, but their meanings are different.}.
\end{definition}

\noindent The following definition is taken from \cite{Balbes-D}.

\begin{definition}
Let $X$ be a topological space. We say that $A\subseteq X$ is \textbf{%
fundamental} if

\begin{description}
\item i) $A$ is a non-empty and open-compact set, or

\item ii) $A=\emptyset $ and for every non-empty collection $\mathcal{A}$ of
non-empty open-compact sets whose intersection is empty, there exists a
finite subcollection of $\mathcal{A}$ with empty intersection.
\end{description}

\noindent We denote $\mathfrak{F}\left( X\right) $ the collection of
fundamental subsets of $X$.
\end{definition}

\noindent Notice that $\emptyset $ is fundamental if every non-empty
collection of open-compact sets with the finite intersection\ property has
non-empty intersection.

A map $f:X\rightarrow Y$ between topological spaces is \textit{strongly
continuous} if it is continuous and the inverse image of a fundamental
subset of $Y$ is a fundamental subset of $X.$\footnote{%
This definition coincides with the one given in \cite{Balbes-D} for the
bounded Balbes-Dwinger spaces.}

Recall that if $X$ is a preordered set, the \textit{Alexandroff (or upper
sets) topology} on $X$ is the topology generated by $\left\{ \uparrow x:x\in
X\right\} ,$ where $\uparrow x=\left\{ y\in X:y\geq x\right\} .$ Notice that 
$\uparrow x$ is an open-compact set in this topological space, thus, every
totally ordered set with its Alexandroff topology is a coherent space.

We present now the topological characterization of the Balbes-Dwinger spaces
given in \cite{Balbes-D}:

\begin{theorem}
A topological space is a Balbes-Dwinger space if, and only if, it is $T_{0},$
coherent and the following condition is satisfied: For every pair of
non-empty collections $\mathcal{A}$ and $\mathcal{B}$ of non-empty
open-compact sets such that $\bigcap\limits_{A\in \mathcal{A}}A\subseteq
\bigcup\limits_{B\in \mathcal{B}}B$, there exist finite subcollections $%
\mathcal{A}_{1}$ of $\mathcal{A}$ and $\mathcal{B}_{1}$ of $\mathcal{B}$
such that $\bigcap\limits_{A\in \mathcal{A}_{1}}A\subseteq
\bigcup\limits_{B\in \mathcal{B}_{1}}B.$\bigskip 
\end{theorem}

\subsection{Balbes-Dwinger duality}

Let $\mathfrak{D}_{p}$ be the category of distributive lattices and proper
homomorphisms and let $\mathfrak{BD}$ be the category of Balbes-Dwinger
spaces and strongly continuous functions. We denote $\mathfrak{D}_{0}^{1}$
the full subcategory of $\mathfrak{D}_{p}$ whose objects are the
distributive bounded lattices and $\mathfrak{S}$ the full subcategory of $%
\mathfrak{BD}$ whose objects are the spectral spaces. If for each morphism $%
\alpha $ in $\mathfrak{D}_{p}$ we define $\mathfrak{spec}(\alpha )=\alpha
^{\ast }$ and for each morphism $f$ in $\mathfrak{BD}$ we define $\mathfrak{F%
}(f)=f^{\ast }$, we have that $\mathfrak{spec}:\mathfrak{D}_{p}\rightarrow 
\mathfrak{BD}$ and $\mathfrak{F}:\mathfrak{BD}\rightarrow \mathfrak{D}_{p}$
are contravariant functors. The following theorem is taken from \cite{Acosta}
and is an extension of a result in \cite{Balbes-D}.

\begin{theorem}
\label{spec es equivalencia}The functors $\mathfrak{spec}:\mathfrak{D}%
_{p}\rightarrow \mathfrak{BD}$ and $\mathfrak{F}:\mathfrak{BD}\rightarrow 
\mathfrak{D}_{p}$ are co-equivalences of categories such that $\mathfrak{spec%
}\circ \mathfrak{F}\cong 1_{\mathfrak{BD}}$ and $\mathfrak{F}\circ \mathfrak{%
spec}\cong 1_{\mathfrak{D}_{p}}.$ The restrictions of these functors to the
categories $\mathfrak{D}_{0}^{1}$ and $\mathfrak{S}$ are also
co-equivalences.
\end{theorem}

In particular, we have that for every distributive lattice $L,$ $\mathfrak{F}%
\left( \mathfrak{spec}\left( L\right) \right) $ is isomorphic to $L$ and,
for every Balbes-Dwinger space $X,$ $\mathfrak{spec}\left( \mathfrak{F}%
\left( X\right) \right) $ is homeomorphic to $X.$

\section{RA-spaces}

We introduce here the notion of RA-space. For each RA-space $X$ we define a
map $h_{X}$ which allows us to characterize some topological properties of $%
X.$ This family of maps will become a natural transformation in Section 6
below.

\begin{definition}
We say that a topological space $X$ is an \textbf{RA-space} if $X$ is
coherent and $\mathfrak{F}(X)$ is a sub-lattice of $\mathcal{\wp }(X).$
\end{definition}

\noindent Notice that $\mathfrak{F}(X)$ is not a sub-lattice of $\mathcal{%
\wp }(X)$ if, and only if, $\emptyset $ is not fundamental and there exist
two non-empty open-compact disjoint sets.\bigskip

\noindent From now on, $X$ will be always an RA-space.

We know that $\mathfrak{spec}\left( \mathfrak{F}\left( X\right) \right) $ is
a Balbes-Dwinger space, hence, by Theorem \ref{spec es equivalencia}, $%
\mathfrak{F}(\mathfrak{spec}\left( \mathfrak{F}(X)\right) )\cong \mathfrak{F}%
(X)$ and thus, $\mathfrak{spec}\left( \mathfrak{F}\left( X\right) \right) $
is an RA-space.

The proof of the following proposition is straightforward:

\begin{proposition}
For each $x\in X$ the set $\left\{ F\in \mathfrak{F}\left( X\right) :x\notin
F\right\} $ is a prime ideal of $\mathfrak{F}\left( X\right) .$
\end{proposition}

Hence, we have a map 
\begin{equation*}
h_{X}:X\rightarrow \mathfrak{spec}\left( \mathfrak{F}\left( X\right) \right)
:x\mapsto \left\{ F\in \mathfrak{F}\left( X\right) :x\notin F\right\} .
\end{equation*}

\begin{proposition}
\label{hx es fuertemente continua}$h_{X}$ is a strongly continuous and open
on its image function.
\end{proposition}

\begin{proof}
\ \ \ 

Take $F\in \mathfrak{F}\left( X\right) -\left\{ \emptyset
\right\} .$

$%
\begin{array}{ll}
x\in \left( h_{X}\right) ^{\ast }\left( d\left( F\right) \right) & 
\Leftrightarrow h_{X}\left( x\right) \in d\left( F\right) \\ 
& \Leftrightarrow F\notin h_{X}\left( x\right) \\ 
& \Leftrightarrow x\in F.%
\end{array}%
$

Thus $\left( h_{X}\right) ^{\ast }\left( d\left( F\right) \right) =F.$ As $%
\mathfrak{F}(\mathfrak{spec}\left( \mathfrak{F}(X)\right) )\cong \mathfrak{F}(X)$, $h_{X}$ is strongly
continuous and open over its image. \qedhere 
\end{proof}

\begin{proposition}
$h_{X}$ is injective if and only if $X$ is $T_{0}.$
\end{proposition}

\begin{proof}
\ \ 
It is enough to remark that, since $X$ is coherent, $h_{X}\left( x\right) =h_{X}\left( y\right)$
is equivalent to  $\overline{\left\{ x\right\} }=\overline{\left\{ y\right\} } $. \QED 

\end{proof}

\begin{proposition}
\label{h sobre ssi X casi-sobrio}$h_{X}$ is surjective if and only if $X$ is
almost-sober.
\end{proposition}

\begin{proof}
\ \ 

\begin{enumerate}
\item Suppose that $h_{X}$ is surjective and we have to see that $X$ is
almost-sober. Let $G$ be a proper irreducible closed set of $X;$ by
definition $G\neq \emptyset .$

\noindent We call $A=X-G.$ We have that $A\neq \emptyset ,$ $A\neq X$
and $A$ is a prime open set of $X.$

\noindent Define $\mathfrak{I}=\left\{ F\in \mathfrak{F}\left( X\right)
:F\subseteq A\right\} .$ As $X$ is coherent, $\mathfrak{I}\neq
\emptyset $ because $A\neq \emptyset $ and $\mathfrak{I}\neq \mathfrak{F}%
\left( X\right) $ given that $A\neq X.$ Since $A$ is a prime open set, $%
\mathfrak{I}$ is a prime ideal of $\mathfrak{F}\left( X\right) ;$ thus, by
the hypothesis, there exists $x\in X$ such that $h_{X}\left( x\right) =%
\mathfrak{I}.$

\noindent We have to see that $G=\overline{\left\{ x\right\} }$:

\noindent $\subseteq :$ If $y\notin \overline{\left\{ x\right\} }$ then
there exists $F\in \mathfrak{F}\left( X\right) $ such that $y\in F$ and $%
x\notin F,$ thus $y\in F$ and $F\in h_{X}\left( x\right) =\mathfrak{I}.$
Therefore, $y\in A$ and $y\notin G.$

\noindent $\supseteq :$ If $y\notin G$ then $y\in A$ and therefore, there
exists $F\in \mathfrak{I}$ such that $y\in F,$ because $X$ is coherent.
Thus, $F\in h_{X}\left( x\right) $ so that $x\notin F.$ Hence, $y\notin 
\overline{\left\{ x\right\} }.$

\item Suppose that $X$ is almost-sober. We have to see that $h_{X}$ is
surjective. Consider $\mathfrak{I}\in \mathfrak{spec}(\mathfrak{F}\left( X\right)) .$

\begin{description}
\item[Case 1:] $\mathfrak{I}=\left\{ \emptyset \right\} .$ We have that $%
\emptyset $ is fundamental. As $\mathfrak{I}$ is a prime ideal,  every
finite intersection of elements of $\mathfrak{F}\left( X\right) -\left\{
\emptyset \right\} $ is non-empty and therefore, $\bigcap\left(
\mathfrak{F}(Z)-\{\emptyset \}\right)\neq \emptyset.$ For each $x\in
\bigcap\left(\mathfrak{F}(Z)-\{\emptyset \}\right)$  we have $h_{X}\left(
x\right) =\mathfrak{I}.$

\item[Case 2:] $\mathfrak{I}\neq \left\{ \emptyset \right\} .$ Define $%
A=\bigcup \mathfrak{I}=\underset{F\in \mathfrak{I}}{\bigcup }F.$ We have
that $A\neq \emptyset .$ If $A=X,$ consider $H\in \mathfrak{F}\left(
X\right) ,$ then $H\subseteq \underset{F\in \mathfrak{I}}{\bigcup }F$ and as 
$H$ is compact, there exist $F_{1},...,F_{n}\in \mathfrak{I}$ such that $%
H\subseteq F_{1}\cup ...\cup F_{n}$ and $F_{1}\cup ...\cup F_{n}\in 
\mathfrak{I},$ so that $H\in \mathfrak{I}.$ Therefore, $\mathfrak{F}\left(
X\right) =\mathfrak{I}$ which contradicts that $\mathfrak{I}$ is a prime
ideal.

\item \noindent We have to show that $A$ is a prime open set. In fact, let $B,C$
be open sets such that $B\cap C\subseteq A.$

\item \noindent As $X$ is coherent, $B=\underset{i}{\bigcup }H_{i}$ and $%
C=\underset{j}{\bigcup }G_{j},$ where $H_{i},G_{j}\in \mathfrak{F}\left(
X\right) $ for each $i$ and each $j.$ Thus,%
\begin{equation*}
B\cap C=\left( \underset{i}{\bigcup }H_{i}\right) \cap \left( \underset{j}{%
\bigcup }G_{j}\right) =\underset{j}{\bigcup }\left( \underset{i}{\bigcup }%
\left( H_{i}\cap G_{j}\right) \right) \subseteq A,
\end{equation*}%
so $H_{i}\cap G_{j}\subseteq A$ for each $i$ and each $j.$ As $H_{i}\cap
G_{j}$ is compact, there exist $F_{1},...,F_{n}\in \mathfrak{I}$ such that $%
H_{i}\cap G_{j}\subseteq F_{1}\cup ...\cup F_{n}$ and $F_{1}\cup ...\cup
F_{n}\in \mathfrak{I},$ then $H_{i}\cap G_{j}\in \mathfrak{I},$ for each $i$
and each $j.$ As $\mathfrak{I}$ is prime, $H_{i}\in \mathfrak{I}$ or $%
G_{j}\in \mathfrak{I},$ for each $i$ and each $j.$ Suppose that $%
G_{j_{0}}\notin \mathfrak{I},$ then $H_{i}\cap G_{j_{0}}\in \mathfrak{I}$
for each $i$ and then, $H_{i}\in \mathfrak{I}$ for every $i;$ thus, $%
B\subseteq A$. Similarly, if $H_{i_{0}}\notin \mathfrak{I}$, we have that  $C\subseteq A.$ We conclude that $G=X-A$ is a proper
irreducible closed and therefore there exists $x\in X$ such that $G=%
\overline{\left\{ x\right\} }.$%
\begin{equation*}
F\in h_{X}(x)\Leftrightarrow x\notin F\Leftrightarrow \overline{\{x\}}\cap
F=\emptyset \Leftrightarrow G\cap F=\emptyset \Leftrightarrow F\subseteq
A\Leftrightarrow F\in \mathfrak{I}.
\end{equation*}%
Hence, $h_{X}\left( x\right) =\mathfrak{I}.$ \qedhere
\end{description} 
\end{enumerate}
\end{proof}

\begin{corollary}
\label{h es homeo sii X es F-esp, To, casi-sobr}$h_{X}$ is a homeomorphism
if and only if $X$ is $T_{0}$ and almost-sober.
\end{corollary}

\section{Almost-spectral spaces}

\noindent In this section we characterize almost-spectral spaces, and show,
among other things, that they are precisely the sober Balbes-Dwinger spaces.

\begin{lemma}
\label{Lema}If $f:X\rightarrow Y$ is continuous and $F$ is an irreducible
closed set of $X$ then $\overline{f(F)}$ is an irreducible closed set of $Y.$
\end{lemma}

\begin{proof}
Let $H$ and $K$ be two closed sets of $Y$ such that $\overline{f(F)}=H\cup K.
$ We have that $F=\left( F\cap f^{-1}(H)\right) \cup \left( F\cap
f^{-1}(K)\right) $ and as $F$ is irreducible, then $F\subseteq f^{-1}(H)$ or 
$F\subseteq f^{-1}(K).$ Hence $f(F)\subseteq H$ or $f(F)\subseteq K$ and
thus, $\overline{f(F)}\subseteq H$ or $\overline{f(F)}\subseteq K$.
Therefore $\overline{f(F)}$ is irreducible.
\end{proof}

This Lemma follows immediately if we work in terms of localic maps or frame
homomorphisms (see \cite{Johnstone}).

\begin{proposition}
If $X$ is a sober space and $Z$ is an open subspace of $X$ then $Z$ is sober.
\end{proposition}

\begin{proof}
Let $F$ be an irreducible closed set of $Z.$ If $i:Z\rightarrow X$ is the
inclusion function then, by Lemma \ref{Lema}, $\overline{i(F)}^{X}$ is an irreducible closed of $X,
$ where $\overline{i(F)}^{X}$ is the closure of $i(F)$ in $X$. As $X$ is
sober, there exists $x\in X$ such that $\overline{F}^{X}=\overline{i(F)}^{X}=%
\overline{\{x\}}^{X}.$ It is clear that $x\in F$ and hence $\overline{\{x\}}%
^{Z}=F$, because the uniqueness is a consequence ot the $T_{0}$ property of $Z$.
\end{proof}

\begin{proposition}
Every almost-spectral space is sober.
\end{proposition}

\begin{proof}
Let $A$ be a commutative ring. We know, by Theorem \ref{nilcompactacion}, that $Spec\left( A\right) $ is an
open subspace of $Spec\left( Q(A)\right) $  and 
$Spec\left( Q(A)\right) $ is sober because it is a spectral space. 
\end{proof}

The following lemma is taken from \cite{Acosta}:

\begin{lemma}
A distributive lattice $L$ has a least element if, and only if, $\mathfrak{%
spec}(L)$ is a sober space.
\end{lemma}

\begin{theorem}
\label{C-esp es B-D sobrio}Every almost-spectral space is a sober
Balbes-Dwinger space.
\end{theorem}

\begin{proof}
Let $A$ be a commutative ring and let $\mathfrak{F}$ be the (distributive)
lattice of the open-compact sets of $Spec\left( A\right) .$ Since $\mathfrak{F}$ has a least element 
we have that $\mathfrak{spec}\left(\mathfrak{F}\right)$ is a Balbes-Dwinger sober space. We have to see that 
 $\mathfrak{spec}\left(\mathfrak{F}\right)$ and  $Spec\left(A\right)$ are homeomorphic.

\noindent If $I$ is a prime ideal of $A,$ define $f(I)=\{B\in \mathfrak{F}%
:I\notin B\}.$ We have to show that $f(I)$ is a prime ideal of $\mathfrak{F}$:

\noindent It is clear that $\emptyset \in f(I)$. As $I$ is a proper ideal of 
$A,$ there exists $a\in A-I,$ then $I\in D(a)$ which is an open-compact set
of $Spec\left( A\right) .$ Hence $f(I)\neq \mathfrak{F}$. If $B,C\in f\left(
I\right) $ we have that $I\notin B\cup C,$ then $B\cup C\in f\left( I\right)
.$ If $B\in f(I)$ and $C\in \mathfrak{F}$ is such that $C\subseteq B,$ we
have that $I\notin C,$ therefore $C\in f\left( I\right) .$ Consider now $%
B,C\in \mathfrak{F}$ such that $B\cap C\in f\left( I\right) .$ We have that $%
I\notin B\cap C$ then$\ I\notin B$ or $I\notin C,$ thus $B\in f\left(
I\right) $ or $C\in f\left( I\right) .$

\noindent Let $\mathcal{J}$ be a prime ideal of $\mathfrak{F}.$ We have to see
that $W=\bigcup\limits_{B\in \mathcal{J}}B$ is a prime open set of $%
Spec\left( A\right) .$ As $\mathcal{J}$ is proper, there exists $B\in 
\mathfrak{F}$ such that $B\notin \mathcal{J}$. Suppose that $B\subseteq W.$
As $B$ is compact, there exist $B_{1},\cdots ,B_{n}\in \mathcal{J}$ such
that $B\subseteq B_{1}\cup \cdots \cup B_{n},$ then $B\in \mathcal{J}$. We
conclude that $W\neq Spec\left( A\right) .$ Let $S$ and $T$ be open sets of $%
Spec\left( A\right) $ such that $S\cap T\subseteq W.$ There exist $%
X,Y\subseteq A$ such that $S=\bigcup\limits_{x\in X}D(x)$ and $%
T=\bigcup\limits_{y\in Y}D(y).$ Thus, $D(xy)=D(x)\cap D(y)\subseteq W,$ for
all $x\in X$ and for all $y\in Y.$ As $D(xy)$ is compact, there exist $%
B_{1},\cdots ,B_{n}\in \mathcal{J}$ such that $D(xy)\subseteq B_{1}\cup
\cdots \cup B_{n},$ therefore $D(x)\cap D(y)=D(xy)\in \mathcal{J}$. As $%
\mathcal{J}$ is prime, $D(x)\in \mathcal{J}$ or $D(y)\in \mathcal{J}.$ If $%
D(y_{0})\notin \mathcal{J}$ for some $y_{0}\in Y$ then $D(x)\in \mathcal{J}$
for all $x\in X,$ therefore $S\subseteq W.$ Similarly, if $D(x_{0})\notin 
\mathcal{J}$ for some $x_{0}\in X,$ then $T\subseteq W.$ We conclude that $W$
is a prime open set of $Spec\left( A\right) .$ Hence $W^{c}$ is an
irreducible closed set of $Spec\left( A\right) $ and as this space is sober,
there exists a unique $P\in Spec\left( A\right) $ such that $\overline{\{P\}}%
=W^{c}.$ Define $g(\mathcal{J})=P.$

\noindent Thus, we have the maps $f:Spec\left( A\right) \rightarrow \mathfrak{spec}\left( \mathfrak{F}%
\right) $ and $g:\mathfrak{spec}\left( \mathfrak{F}\right) \rightarrow Spec\left(
A\right) .$ Besides:%
\begin{equation*}
C\in f\left( g\left( \mathcal{J}\right) \right) \Leftrightarrow g\left( 
\mathcal{J}\right) \notin C\Leftrightarrow C\subseteq \overline{\left\{
g\left( \mathcal{J}\right) \right\} }^{c}\Leftrightarrow C\subseteq
\bigcup\limits_{B\in \mathcal{J}}B\Leftrightarrow C\in \mathcal{J},
\end{equation*}%
where the last equivalence is a consequence of the compactness of $C$.

\noindent On the other hand, as $\bigcup\limits_{B\in
f(I)}B=\bigcup\limits_{I\notin B}B=\overline{\left\{ I\right\} }^{c}$, we
have that $g(f(I))=I.$

\noindent We need to see that $f$ is continuous and open. Consider $K\in 
\mathfrak{F}:$%
\begin{equation*}
I\in f^{-1}(d(K))\Leftrightarrow f\left( I\right) \in d(K)\Leftrightarrow
K\notin f(I)\Leftrightarrow I\in K
\end{equation*}%
then $f^{-1}(d(K))=K.$ We conclude that $f$ is continuous and open over its
image and as the image is $\mathfrak{spec}\left( \mathfrak{F}\right) ,$ $f:Spec\left(
A\right) \rightarrow \mathfrak{spec}\left( \mathfrak{F}\right) $ is a homeomorphism.
\end{proof}

\begin{theorem}
\label{Abto de espect es c-espec}Every open of a spectral space is an
almost-spectral space.
\end{theorem}

\begin{proof}
Let $A$ be a commutative ring with identity and let $Z$ be an open set of $%
Spec\left( A\right) .$ We know that there exists $I$ ideal of $A$ such that $%
Z^{c}=V(I).$ We have, by Theorem \ref{nilcompactacion}, 
\begin{equation*}
Spec(I)\approx \left( V\left( I\right) \right) ^{c}=Z.
\end{equation*}

\noindent Therefore $Z$ is an almost-spectral space.
\end{proof}

\begin{theorem}
\label{Charact almost-spectral}Let $Z$ be a topological space. The following
statements are equivalent:

(i) $Z$ is almost-spectral.

(ii) $Z$ is open-dense of a spectral space.

(iii) $Z$ is open of a spectral space.

(iv) $Z$ is a sober Balbes-Dwinger space.

(v) $Z$ is homeomorphic to the prime spectrum of a distributive lattice with
minimum.
\end{theorem}

\begin{proof}
(i)$\Rightarrow $(ii): If $Z$ is almost-spectral there exists a commutative
ring $A$ such that $Z\approx Spec\left( A\right) .$ We know that $Spec\left(
A\right) \approx Spec_{A}\left( U_{0}\left( A\right) \right) $ which is an
open-dense of $Spec\left( U_{0}\left( A\right) \right) $. (See \cite%
{Acosta-Rubio}).

\noindent (ii)$\Rightarrow $(iii): Trivial.

\noindent (iii)$\Rightarrow $(i): Theorem \ref{Abto de espect es c-espec}.

\noindent (i)$\Rightarrow $(iv): Theorem \ref{C-esp es B-D sobrio}.

\noindent (iv)$\Rightarrow $(v): Theorem 9 (IV-1) of \cite{Balbes-D} and
Proposition 5.7 of \cite{Acosta}.

\noindent (v)$\Rightarrow $(iii): Let $L$ be a distributive lattice with $0$
such that $Z\approx \mathfrak{spec}\left( L\right) .$ We have that $\mathfrak{spec}\left( Z\right) 
$ is an open of $\mathfrak{spec}\left( \widehat{L}\right) $, where $\widehat{L}=L\Lsh
\Theta $ and $\Theta $ is a lattice with only one element\footnote{%
If $L$ and $M$ are lattices, its \textit{ordinal sum} $L\Lsh M$ is defined
by the set $L\times \left\{ 0\right\} \cup M\times \left\{ 1\right\} $
ordered by: $\left( x,i\right) \leq \left( y,j\right) $ if $x\leq y$ and $i=j
$ or if $i=0$ and $j=1.$}.
\end{proof}

\begin{proposition}
If $X$ is spectral then every open subspace of $X$ is almost-spectral and
every closed subspace is spectral.
\end{proposition}

\begin{proof}
The first part is consequence of Theorem \ref{Abto de espect es c-espec}.
Let $Z$ be a closed subspace of $X.$ As $X\approx Spec\left( A\right) ,$
where $A$ is a ring with identity, then $Z\approx V\left( I\right) \approx
Spec\left( A/I\right) ,$ for some ideal $I$ of $A.$
\end{proof}

\noindent Similarly we obtain the following proposition.

\begin{proposition}
If $Z$ is almost-spectral then every open subspace is almost-spectral and
every closed subspace is almost-spectral.
\end{proposition}

\section{Up-spectral and down-spectral spaces}

In this section we present spectral versions for the up-spectral and
down-spectral spaces. As a consequence, we obtain a new topological
characterization of the Balbes-Dwinger spaces.

First of all we recall the definition of these kind of topological spaces:

\begin{definition}
A space is \textbf{up-spectral} if it is coherent and sober. A space is 
\textbf{down-spectral} if it is $T_{0}$, coherent, compact and almost-sober.
(See \cite{Echi 2}).
\end{definition}

Actually, the notions of up-spectral space and almost-spectral space are
equivalent, as the following theorem shows.

\begin{theorem}
\label{up-espect ento B-D + sobrio}Let $Z$ be a topological space. The
following statements are equivalent:

(i) $Z$ is up-spectral.

(ii) $Z$ is almost-spectral.
\end{theorem}

\begin{proof}
If $Z$ is up-spectral then $Z^{\omega }$ (the
trivial compactification of $Z$) is a spectral space (Proposition 1.5 of 
\textit{\cite{Echi 1})}. Thus, $Z$ is open of a spectral space and therefore
is almost-spectral. Reciprocally, if $Z$ is almost-spectral then $Z$ is a Balbes-Dwinger 
and sober space. Hence, $Z$ is up-spectral.
\end{proof}\bigskip

\begin{corollary}
Let $Z$ be a topological space. The following statements are equivalent:

(i) $Z$ is up-spectral.

(ii) $Z$ is homeomorphic to the prime spectrum of a distributive lattice
with minimum.
\end{corollary}

\begin{theorem}
Every Balbes-Dwinger space is almost-sober.
\end{theorem}

\begin{proof}
Let $Z$ be a Balbes-Dwinger space. Let $G$ be a proper irreducible closed
set of $Z.$ Then $A=Z-G$ is a non-empty prime open set of $Z.$ So we have
that $A=\bigcup\limits_{i\in \Lambda }F_{i}$ for some collection of
non-empty open-compact sets of $Z.$ Let $\mathfrak{I}$ be the ideal of $%
\mathfrak{F}(Z)$ generated by $\{F_{i}\}_{i\in \Lambda },$ $\mathfrak{I}%
=\left\{ F\in \mathfrak{F}(Z):F\subseteq A\right\} .$ As $A$ is a prime open
set, $\mathfrak{I}$ is a prime ideal of $\mathfrak{F}(Z).$ Since $Z$ is a
Balbes-Dwinger space, there exists $x\in Z$ such that $\mathfrak{I}=\{F\in 
\mathfrak{F}(Z):x\notin F\}.$ It is clear that $G=\overline{\{x\}}.$
\end{proof}

\noindent The following theorem gives an additional and simpler topological
characterization for the Balbes-Dwinger spaces.

\begin{theorem}
Let $Z$ be a topological space. The following statements are equivalent:

(i) $Z$ is $T_{0},$ coherent and almost-sober.

(ii) $Z$ is a Balbes-Dwinger space.
\end{theorem}

\begin{proof}
By the previous theorem, (ii) implies (i). Now, let $Z$ be a $T_{0},$
coherent and almost-sober space. Suppose that there exist $F,$ $G$ non-empty open-compact sets such
that $F\cap G=\emptyset ,$ so $F^{c}\cup G^{c}=Z$ and then, $Z$ is not an
irreducible set. As $Z$ is $T_{0}$ and almost-sober, then $Z$ is sober and
therefore up-spectral. Hence, by Theorems \ref{Charact almost-spectral} and \ref{up-espect ento B-D +
sobrio}, $Z$ is a Balbes-Dwinger space.

\noindent If there do not exist non-empty open-compact sets $F$ and $G$ such that $%
F\cap G=\emptyset ,$ then $\mathfrak{F}(Z)$ is a distributive lattice,
because $Z$ is coherent. Therefore, by Theorem 2 (IV) of \cite{Balbes-D}, we have
that $\mathfrak{spec}(\mathfrak{F}(Z))$ is a Balbes-Dwinger space. On the
other hand, as $Z$ is an almost-sober, $T_{0}$ F-space, then by the
Corollary \ref{h es homeo sii X es F-esp, To, casi-sobr}, $Z$ and $\mathfrak{spec}%
(\mathfrak{F}(Z))$ are homeomorphic, thus $Z$ is a Balbes-Dwinger space.
\end{proof}

As a corollary we obtain the spectral version of the down-spectral spaces.

\begin{corollary}
Let $Z$ be a topological space. The following statements are equivalent:

(i) $Z$ is down-spectral.

(ii) $Z$ is a Balbes-Dwinger and compact space.

(iii) $Z$ is homeomorphic to the prime spectrum of a distributive lattice
with maximum.
\end{corollary}

\section{An extension of the Balbes-Dwinger duality}

\noindent We denote $\mathfrak{FSp}$ the category whose objects are the
RA-spaces and whose morphisms are the strongly continuous functions.\bigskip

\begin{definition}
For each strongly continuous function $f:X\rightarrow Y$ between RA-spaces
we define $\mathfrak{F}\left( f\right) :\mathfrak{F}\left( Y\right)
\rightarrow \mathfrak{F}\left( X\right) $ by $\mathfrak{F}\left( f\right)
\left( F\right) =f^{\ast }\left( F\right) $ and for each proper homomorphism 
$h:L\rightarrow M$ between distributive lattices we define $\mathfrak{spec}%
\left( h\right) :\mathfrak{spec}\left( M\right) \rightarrow \mathfrak{spec}%
\left( L\right) $ by $\mathfrak{spec}\left( h\right) \left( I\right)
=h^{\ast }\left( I\right) .$
\end{definition}

\begin{lemma}
If $f:X\rightarrow Y$ is a strongly continuous function between RA-spaces,
then $\mathfrak{F}\left( f\right) :\mathfrak{F}\left( Y\right) \rightarrow 
\mathfrak{F}\left( X\right) $ is a proper homomorphism.
\end{lemma}

\begin{proof}
Let $\mathfrak{I}$ be a prime ideal of $\mathfrak{F}\left( X\right) .$ If $%
\mathfrak{F}\left( f\right) ^{\ast }\left( \mathfrak{I}\right) =\mathfrak{F}%
\left( Y\right) $ then for every $F\in \mathfrak{F}\left( Y\right) $ we have
that $f^{\ast }(F)\in \mathfrak{I}$. Thus, 
\begin{equation*}
X=f^{\ast }\left( \bigcup\limits_{F\in \mathfrak{F}(Y)}F\right)
=\bigcup\limits_{F\in \mathfrak{F}(Y)}f^{\ast }(F)\subseteq
\bigcup\limits_{G\in \mathfrak{I}}G.
\end{equation*}%
\noindent As $\mathfrak{I}$ is a proper ideal of $\mathfrak{F}(X),$ there
exists $G_{0}\in \mathfrak{F}(X)-\mathfrak{I}$. Then $G_{0}\subseteq
\bigcup\limits_{G\in \mathfrak{I}}G$ and as $G_{0}$ is compact, there exist 
$G_{1},\ldots G_{n}\in \mathfrak{I}$ such that $G_{0}\subseteq
\bigcup\limits_{i=1}^{n}G_{i}\in \mathfrak{I}$ and therefore $G_{0}\in 
\mathfrak{I}$, which is absurd. The missing details to see that $\mathfrak{F}%
\left( f\right) ^{\ast }\left( \mathfrak{I}\right) $ is a prime ideal of $%
\mathfrak{F}\left( Y\right) $ are obtained directly from the definition of $%
\mathfrak{F}\left( f\right) .$
\end{proof}

\begin{lemma}
If $h:L\rightarrow M$ is a proper homomorphism between distributive
lattices, then $\mathfrak{spec}\left( h\right) :\mathfrak{spec}\left(
M\right) \rightarrow \mathfrak{spec}\left( L\right) $ is a strongly
continuous function.
\end{lemma}

\begin{proof}
By the Proposition 5.6 of \cite{Acosta} we know that $\mathfrak{spec}\left( h\right) $
sends open-compact sets to open-compact sets by inverse image. We need to
see that if $\emptyset $ is fundamental in $\mathfrak{spec}\left( L\right) $ then $%
\emptyset $ is fundamental in $\mathfrak{spec}\left( M\right) ,$ but this is equivalent
to see that if $L$ has minimum, then $M$ has minimum. (Proposition 5.8 of 
\cite{Acosta}).

\noindent We call $0$ the minimum of $L$ and suppose that $M$ has not
minimum. If $s=h\left( 0\right) ,$ there exists $t\in M$ such that $t<s.$
We call $I$ to the ideal generated by $t$ and $F$ to the filter
generated by $s.$ As $M$ is distributive, there exists a prime ideal $P$ of $%
M$ such that $I\subseteq P$ and $P\cap F=\emptyset .$ As $h$ is a proper
homomorphism, then $h^{\ast }\left( P\right) =\emptyset $ is a prime ideal
of $L,$ but this is contradictory.
\end{proof}

\noindent It is easy to check that $\mathfrak{F}:\mathfrak{FSp}\rightarrow 
\mathfrak{D}_{p}$ and $\mathfrak{spec}:\mathfrak{D}_{p}\rightarrow \mathfrak{%
FSp}$ are contravariant functors.\bigskip

\noindent The following theorem extends Theorem \ref{spec es equivalencia}.

\begin{theorem}
The functors $\mathfrak{F}$ and $\mathfrak{spec}$ are right adjoint
contravariant functors.
\end{theorem}

\begin{proof}
Let $M$ be a distributive lattice and let $X$ be a
RA-space. If $\alpha
:M\rightarrow \mathfrak{F}(X)$ is a proper
homomorphism, then $\mathfrak{spec}\left(
\alpha \right) :\mathfrak{spec}%
\left( \mathfrak{F}(X)\right) \rightarrow \mathfrak{spec}\left(
M\right) $
is a strongly continuous function and it is known that $%
h_{X}:X%
\rightarrow \mathfrak{spec}\left( \mathfrak{F}(X)\right) $ also is a
strongly continuous function (Proposition \ref{hx es fuertemente continua}%
). We have to
see that $\lambda _{\left( M,X\right) }:\left[ M,\mathfrak{F}(X)%
\right] _{%
\mathcal{D}_{p}}\rightarrow \left[ X,\mathfrak{spec}\left(
M\right) \right] _{\mathfrak{FSp}}$ defined by $\lambda _{\left( M,X\right)
}\left( \alpha \right)
=\mathfrak{spec}\left( \alpha \right) \circ h_{X}$
is a bijective function.

\noindent i) $\lambda _{\left( M,X\right) }$ is
injective:

\noindent $%
\begin{array}{l}
\lambda _{\left( M,X\right)
}\left( \alpha \right) =\lambda _{\left(
M,X\right) }\left( \beta \right)
\\   \Leftrightarrow \mathfrak{spec}\left( \alpha \right)
\circ h_{X}=%
\mathfrak{spec}\left( \beta \right) \circ h_{X} \\ 
 \Leftrightarrow \alpha
^{\ast }\circ h_{X}=\beta ^{\ast }\circ h_{X} \\ 
 \Leftrightarrow \alpha
^{\ast }\left( h_{X}\left( x\right) \right) =\beta
^{\ast }\left(
h_{X}\left( x\right) \right) ,\text{ }\forall x\in X \\ 
 \Leftrightarrow %
\left[ z\in \alpha ^{\ast }\left( h_{X}\left( x\right)
\right)
\Leftrightarrow z\in \beta ^{\ast }\left( h_{X}\left( x\right)
\right) %
\right] ,\text{ }\forall z\in M,\text{ }\forall x\in X \\ 
 \Leftrightarrow %
\left[ \alpha \left( z\right) \in h_{X}\left( x\right)
\Leftrightarrow
\beta \left( z\right) \in h_{X}\left( x\right) \right] ,%
\text{ }\forall
z\in M,\text{ }\forall x\in X \\ 
 \Leftrightarrow \left[ x\notin \alpha
\left( z\right) \Leftrightarrow
x\notin \beta \left( z\right) \right] ,%
\text{ }\forall z\in M,\text{ }%
\forall x\in X \\ 
 \Leftrightarrow
\alpha \left( z\right) =\beta \left( z\right) ,\text{ }%
\forall z\in M \\

 \Leftrightarrow \alpha =\beta .%
\end{array}%
$

\noindent ii) $%
\lambda _{\left( M,X\right) }$ is surjective: Let be $\varepsilon
:X%
\rightarrow \mathfrak{spec}\left( M\right) $ a strongly continuous map$.$%

We have that $\mathfrak{F}\left( \varepsilon \right) :\mathfrak{F}\left(
\mathfrak{spec}\left( M\right) \right) \rightarrow \mathfrak{F}(X)$ is a 
proper homomorphism.
Consider the proper homomorphism $d:M\rightarrow 
\mathfrak{F}\left(
\mathfrak{spec}\left( M\right) \right) $  (Theorem 5.7
of 
\cite{Acosta}).

\noindent $\lambda _{\left( M,X\right) }\left( 
\mathfrak{F}\left(
\varepsilon \right) \circ d\right) =\mathfrak{spec}%
\left( \mathfrak{F}\left(
\varepsilon \right) \circ d\right) \circ
h_{X}=\left( \mathfrak{F}\left(
\varepsilon \right) \circ d\right) ^{\ast
}\circ h_{X}.$ Let be $x\in X.$

\noindent $%
\begin{array}{ll}
I\in
\left( \mathfrak{F}\left( \varepsilon \right) \circ d\right) ^{\ast
}\circ
h_{X}\left( x\right) & \Leftrightarrow \left( \mathfrak{F}%
\left(
\varepsilon \right) \circ d\right) \left( I\right) \in h_{X}\left(
x\right)
\\ 
& \Leftrightarrow x\notin \left( \mathfrak{F}\left(
\varepsilon \right)
\circ d\right) \left( I\right) =\varepsilon ^{\ast
}\left( d\left( I\right)
\right) \\ 
& \Leftrightarrow \varepsilon \left(
x\right) \notin d\left( I\right) \\ 
& \Leftrightarrow I\in \varepsilon
\left( x\right) .%
\end{array}%
$

\noindent Therefore, $\lambda
_{\left( M,X\right) }\left( \mathfrak{F}\left(
\varepsilon \right) \circ
d\right) =\varepsilon .$

\noindent The family $\lambda =\left\{ \lambda
_{\left( M,X\right) }\right\}
_{\left( M,X\right) \in Ob\mathcal{D}%
_{p}\times Ob\mathfrak{FSp}}$ is a
natual bijection: Let $g\in \left[ Y,X%
\right] _{\mathfrak{FSp}}$ and $\alpha
\in \left[ M,\mathfrak{F}(X)\right] _{%
\mathcal{D}_{p}},$ we need to see that 
$\lambda _{\left( M,Y\right)
}\left( \mathfrak{F}\left( g\right) \circ
\alpha \right) =\lambda _{\left(
M,X\right) }\left( \alpha \right) \circ g.$
Take $y\in Y.$

\noindent $%
\begin{array}{ll}
I\in \lambda _{\left( M,Y\right) }\left( \mathfrak{F}%
\left( g\right) \circ
\alpha \right) \left( y\right) & \Leftrightarrow I\in 
\mathfrak{spec}\left( \mathfrak{F}%
\left( g\right) \circ \alpha \right)
\circ h_{Y}\left( y\right) \\ 
& \Leftrightarrow I\in \left( \mathfrak{F}%
\left( g\right) \circ \alpha
\right) ^{\ast }\circ h_{Y}\left( y\right) \\

& \Leftrightarrow \left( \mathfrak{F}\left( g\right) \circ \alpha
\right)
\left( I\right) \in h_{Y}\left( y\right) \\ 
& \Leftrightarrow
y\notin \left( \mathfrak{F}\left( g\right) \circ \alpha
\right) \left(
I\right) \\ 
& \Leftrightarrow y\notin g^{\ast }\left( \alpha \left(
I\right) \right) \\ 
& \Leftrightarrow g\left( y\right) \notin \alpha
\left( I\right) \\ 
& \Leftrightarrow \alpha \left( I\right) \in
h_{X}\left( g\left( y\right)
\right) \\ 
& \Leftrightarrow I\in \alpha
^{\ast }\left( h_{X}\left( g\left( y\right)
\right) \right) \\ 
&
\Leftrightarrow I\in \mathfrak{spec}\left( \alpha \right) \left( h_{X}\left(
g\left(
y\right) \right) \right) .%
\end{array}%
$

\noindent Similarly
it is obtained that for $f\in \left[ L,M\right] _{%
\mathcal{D}_{p}}$ and $%
\alpha \in \left[ M,\mathfrak{F}(X)\right] _{\mathcal{%
D}_{p}},$ it must $%
\lambda _{\left( L,X\right) }\left( \alpha \circ f\right)
=\mathfrak{spec}%
\left( f\right) \circ \lambda _{\left( M,X\right) }\left( \alpha
\right) .$ \qedhere%

\end{proof}

\noindent The co-equivalence of this adjunction is between the categories $%
\mathfrak{D}_{p}$ and $\mathfrak{BD}$.

We introduce here two full subcategories of $\mathfrak{D}_{p}$ and two full
subcategories of $\mathfrak{BD}:$

\begin{tabular}{|l|l|}
\hline
\textbf{Name} & \textbf{Objects} \\ \hline
$\mathfrak{D}_{0}$ & Distributive lattices with minimum \\ \hline
$\mathfrak{D}^{1}$ & Distributive lattices with maximum \\ \hline
$\mathfrak{US}$ & Up-spectral spaces = Almost-spectral spaces \\ 
& = Sober Balbes-Dwinger spaces \\ \hline
$\mathfrak{DS}$ & Down-spectral spaces = Compact Balbes-Dwinger spaces \\ 
\hline
\end{tabular}%
\bigskip

\begin{corollary}
The following pairs of categories are co-equivalent:

(i) $\mathfrak{D}_{0}$ and $\mathfrak{US}$.

(ii) $\mathfrak{D}^{1}$ and $\mathfrak{DS}$.
\end{corollary}

Now, it is clear that the notions of up-spectral space and down-spectral
space are mutually dual in the category $\mathfrak{BD}$.\bigskip

\begin{equation*}
\ast \ast \ast
\end{equation*}

\noindent The following diagram summarizes the previous results.

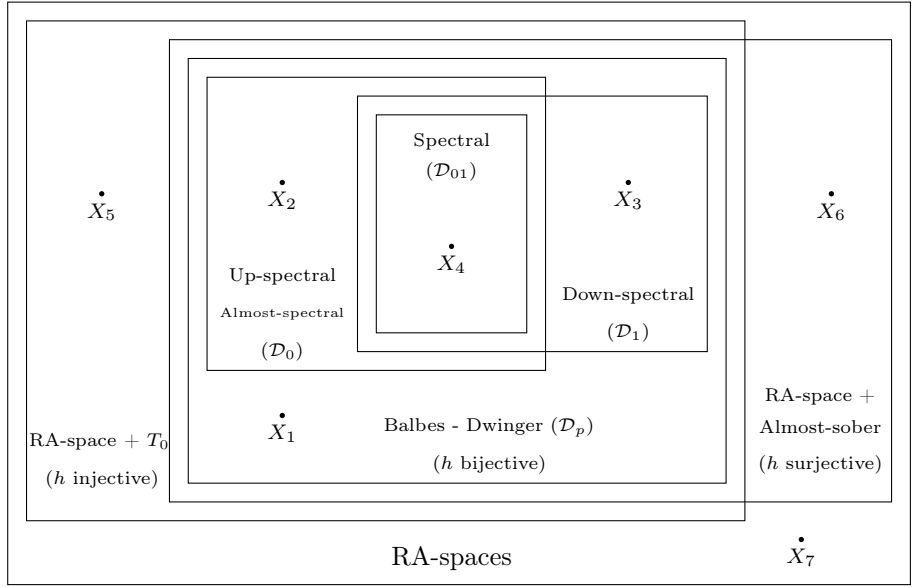
\begin{figure}[!htb]
\centering
\begin{tikzpicture}[scale=0.5]
    \draw (0.7,3.3) rectangle (24.7,18.8); \draw (12.5,4) node{RA-spaces};
    \draw (1.2,5) rectangle (20.3,18.3); \draw (3.2,6.1) node{\scriptsize ($h$ injective)}; 
                                     \draw (3.14,7.1) node{\scriptsize RA-space + $T_{0}$};
    \draw (5,5.5) rectangle (24.2,17.8); \draw (22.3,6.5) node{\scriptsize ($h$ surjective)};
                                      \draw (22.3,7.5) node{\scriptsize Almost-sober};
                                      \draw (22.3,8.3) node{\scriptsize RA-space +};                                     
    \draw (5.5,6) rectangle (19.8,17.3); \draw (13.5,6.5) node{\scriptsize ($h$ bijective)};
                                      \draw (13.5,7.5) node{\scriptsize Balbes - Dwinger ($\mathcal{D}_{p}$)};
    \draw (6,9) rectangle (15,16.8); \draw (8,9.5) node{\scriptsize ($\mathcal{D}_{0}$)};
                                   \draw (8,10.5) node{\tiny Almost-spectral};
                                   \draw (8,11.5) node{\scriptsize Up-spectral};                                   
    \draw (10,9.5) rectangle (19.3,16.3); \draw (17.2,10) node{\scriptsize ($\mathcal{D}_{1}$)};
                                   \draw (17.2,11) node{\scriptsize Down-spectral};                                      
    \draw (10.5,10) rectangle (14.5,15.8); \draw (12.5,14.3) node{\scriptsize ($\mathcal{D}_{01}$)};
                                    \draw (12.5,15.1) node{\scriptsize Spectral};                                          
    \fill (8,7.8) circle (2pt) node[below]{\footnotesize $X_{1}$}
          (8,14) circle (2pt) node[below]{\footnotesize $X_{2}$}
          (17.2,14) circle (2pt) node[below]{\footnotesize $X_{3}$}
          (12.5,12.3) circle (2pt) node[below]{\footnotesize $X_{4}$}
          (3.2,13.7) circle (2pt) node[below]{\footnotesize $X_{5}$}
          (22.6,13.7) circle (2pt) node[below]{\footnotesize $X_{6}$}
          (21.8,4.5) circle (2pt) node[below]{\footnotesize $X_{7}$};

   %\draw (5,-3.9) node[above]{{\rm Figure 1}};
\end{tikzpicture}
\end{figure}

\noindent The represented examples in the diagram are:

\noindent $X_{1}=\mathbb{Z}$ with the Alexandroff topology.

\noindent $X_{2}=\mathbb{Z}^{-}$ with the Alexandroff topology.

\noindent $X_{3}=\mathbb{N}$ with the Alexandroff topology.

\noindent $X_{4}=\left\{ 1,2,3,4\right\} $ with the Alexandroff topology.

\noindent $X_{5}=\mathbb{R}$ with the Alexandroff topology. $X_{5}$ is not
almost-sober because for example, $\left( -\infty ,3\right) $ is a proper
irreducible closed set that is not the closure of any point.

\noindent $X_{6}=\left\{ a,b,c\right\} $ with the topology $\left\{
\emptyset ,\left\{ a,b\right\} ,\left\{ a,b,c\right\} \right\} .$

\noindent $X_{7}=\left( \mathbb{R\times }\left\{ 0\right\} \right) \cup
\left( \mathbb{R\times }\left\{ 1\right\} \right) \cup \left\{ \left( \omega
,0\right) ,\left( \omega ,1\right) \right\} $ with the Alexandroff topology
obtained from the preorder given by: $\left( x,i\right) \leq \left(
y,j\right) $ if $x,y\in \mathbb{R}$, $x\leq y$ and $i=j;$ $\left( x,i\right)
\leq \left( \omega ,j\right) $ for all $x\in \mathbb{R}$ and all $i,j;$ $%
\left( \omega ,0\right) \leq \left( \omega ,1\right) $ and $\left( \omega
,1\right) \leq \left( \omega ,0\right) .$ $X_{7}$ is not almost-sober
because for example, $\mathbb{R\times }\left\{ 0\right\} $ is a proper
irreducible closed set that is not the closure of a point.

\bigskip

\end{document}